\documentclass[10pt]{article}%
\usepackage{amsmath,amssymb}
\usepackage{amsmath}
\usepackage{amsfonts}
\usepackage{mathrsfs}
\usepackage{amssymb}
\usepackage[linkcolor=black,anchorcolor=black,citecolor=black]{hyperref}
\usepackage{graphicx}
\usepackage{float}
\providecommand{\U}[1]{\protect\rule{.1in}{.1in}}
\providecommand{\U}[1]{\protect \rule{.1in}{.1in}}
\newtheorem{theorem}{Theorem}[section]

\newtheorem{definition}[theorem]{Definition}

\newtheorem{remark}[theorem]{Remark}

\newenvironment{proof}[1][Proof]{\noindent \textbf{#1.} }{\  $\Box$}
\begin{document}

\title{A note on characterizations of $G$-normal distribution }
\author{Peng Luo\thanks{School of Mathematics and Qilu Securities Institute for Financial Studies, Shandong University and Department of Mathematics and Statistics, University of Konstanz; pengluo1989@gmail.com}
\and Guangyan Jia\thanks{Qilu Securities Institute for Financial Studies, Shandong University; jiagy@sdu.edu.cn}}
\date{}
\maketitle


\begin{abstract}
In this paper, we show that the $G$-normality of $X$ and $Y$ can be characterized according to the form of $f$ such that the distribution of $\lambda X+f(\lambda)Y$ does not depend on $\lambda$, where $Y$ is an independent copy of $X$ and $\lambda$ is in the domain of $f$. Without the condition that $Y$ is identically distributed with $X$, we still have a similar argument.

\textbf{Keywords}:  $G$-normal distribution, Characterization.
\\
\\
\textbf{Mathematics Subject Classification (2010). }60H30, 60H10.
\end{abstract}

\section{Introduction}
In the classical framework, let $X$ and $Y$ be two independent random vectors and $f$ be a function defined on an interval of $\mathbb{R}$. Nguyen and Sampson (\cite{Ng}) obtained that the distribution of $X$ and $Y$ can be characterized according to the form of $f$ such that the distribution of the random vector $\lambda X+f(\lambda)Y$ does not depend on $\lambda$, where $\lambda$ takes on values in the domain of $f$. These results complement previously obtained characterizations where $X$ and $Y$ are required to be identically distributed and for one value $\lambda^*$, $\lambda^*+f(\lambda^*)Y$ has the same distribution as $X$ which are discussed in Kagan et al. (\cite{Ka}). All these results are related to the Marcinkiewicz theorem (see \cite{Ka}), which simply says that under suitable conditions if $X$ and $Y$ are independent and identically distributed, and $\lambda_1X+\tau_1Y$ and $\lambda_2X+\tau_2Y$ have the same distribution, then $X$ and $Y$ have a normal distribution. Note that, in fact, $\tau_i$, $i=1,2$, can not be arbitrary constants, but must satisfy $\tau_i=\sqrt{1-\lambda_{i}^2}$, $i=1,2$.

Recently, Peng systemically established a time-consistent fully nonlinear
expectation theory (see \cite{Peng1}, \cite{Peng2} and \cite{Peng3}).

As a typical and important case, Peng (2006) introduced the
$G$-expectation theory(see \cite{Peng4} and the references therein).
In the  $G$-expectation framework ($G$-framework for short), the
notion of independence, identically distributed and $G$-normal distribution were established.

Motivated by their works, we obtain several characterizations of $G$-normal distribution. The paper is organized as follow: In section 2, we recall some notations and results that we will use in this paper. In section 3, we obtain our main results.
\section{Preliminaries}
We present some preliminaries in the theory of sublinear
expectation, $G$-normal distribution under $G$-framework. More details can be found in Peng \cite{Peng4}.

\begin{definition}\label{def2.1}
Let $\Omega$ be a given set and let $\mathcal{H}$ be a vector
lattice of real valued functions defined on $\Omega$, namely $c\in \mathcal{H}$
for each constant $c$ and $|X|\in \mathcal{H}$ if $X\in \mathcal{H}$.
$\mathcal{H}$ is considered as the space of random variables. A sublinear
expectation $\mathbb{\hat{E}}$ on $\mathcal{H}$ is a functional $\mathbb{\hat
{E}}:\mathcal{H}\rightarrow \mathbb{R}$ satisfying the following properties:
for all $X,Y\in \mathcal{H}$, we have
\item[(a)] Monotonicity: If $X\geq Y$ then $\mathbb{\hat{E}}[X]\geq
\mathbb{\hat{E}}[Y]$;
\item[(b)] Constant preservation: $\mathbb{\hat{E}}[c]=c$;
\item[(c)] Sub-additivity: $\mathbb{\hat{E}}[X+Y]\leq \mathbb{\hat{E}
}[X]+\mathbb{\hat{E}}[Y]$;

\item[(d)] Positive homogeneity: $\mathbb{\hat{E}}[\lambda X]=\lambda
\mathbb{\hat{E}}[X]$ for each $\lambda \geq 0$.\\
$\left(\Omega,\mathcal{H},\mathbb{\hat{E}}\right)$ is called a sublinear expectation space.
\end{definition}

\begin{definition}
\label{def2.2} Let $X_{1}$ and $X_{2}$ be two $n$-dimensional random vectors
defined respectively in sublinear expectation spaces $\left(\Omega_{1}%
,\mathcal{H}_{1},\mathbb{\hat{E}}_{1}\right)$ and $\left(\Omega_{2},\mathcal{H}%
_{2},\mathbb{\hat{E}}_{2}\right)$. They are called identically distributed, denoted
by $X_{1}\overset{d}{=}X_{2}$, if $\mathbb{\hat{E}}_{1}\left[\varphi(X_{1}%
)\right]=\mathbb{\hat{E}}_{2}\left[\varphi(X_{2})\right]$, for all$\  \varphi \in C_{b.Lip}%
(\mathbb{R}^{n})$, where $C_{b.Lip}(\mathbb{R}^{n})$ denotes the space of
bounded and Lipschitz functions on $\mathbb{R}^{n}$.
\end{definition}

\begin{definition}
\label{def2.3} In a sublinear expectation space $\left(\Omega,\mathcal{H}%
,\mathbb{\hat{E}}\right)$, a random vector $Y=(Y_{1},\cdot \cdot \cdot,Y_{n})$,
$Y_{i}\in \mathcal{H}$, is said to be independent of another random vector
$X=(X_{1},\cdot \cdot \cdot,X_{m})$, $X_{i}\in \mathcal{H}$ under $\mathbb{\hat
{E}}[\cdot]$, denoted by $Y\bot X$, if for every test function $\varphi \in
C_{b.Lip}(\mathbb{R}^{m}\times \mathbb{R}^{n})$ we have $\mathbb{\hat{E}%
}[\varphi(X,Y)]=\mathbb{\hat{E}}\left[\mathbb{\hat{E}}[\varphi(x,Y)]_{x=X}\right]$.
\end{definition}

\begin{definition}
\label{def2.4} ($G$-normal distribution) A $d$-dimensional random vector
$X=(X_{1},\cdot \cdot \cdot,X_{d})$ in a sublinear expectation space
$\left(\Omega,\mathcal{H},\mathbb{\hat{E}}\right)$ is called $G$-normally distributed if
for each $a,b\geq0$ we have
\[
aX+b\bar{X}\overset{d}{=}\sqrt{a^{2}+b^{2}}X,
\]
where $\bar{X}$ is an independent copy of $X$, i.e., $\bar{X}\overset{d}{=}X$
and $\bar{X}\bot X$. Here the letter $G$ denotes the function
\[
G(A):=\frac{1}{2}\mathbb{\hat{E}}[\langle AX,X\rangle]:\mathbb{S}%
_{d}\rightarrow \mathbb{R},
\]
where $\mathbb{S}_{d}$ denotes the collection of $d\times d$ symmetric matrices.
\end{definition}

Peng \cite{Peng4} showed that $X=(X_{1},\cdot \cdot \cdot,X_{d})$ is $G$-normally
distributed if and only if for each $\varphi \in C_{b.Lip}(\mathbb{R}^{d})$,
$u(t,x):=\mathbb{\hat{E}}[\varphi(x+\sqrt{t}X)]$, $(t,x)\in \lbrack
0,\infty)\times \mathbb{R}^{d}$, is the solution of the following $G$-heat
equation:%
\[
\partial_{t}u-G(D_{x}^{2}u)=0,\ u(0,x)=\varphi(x).
\]

The function $G(\cdot):\mathbb{S}_{d}\rightarrow \mathbb{R}$ is a monotonic,
sublinear mapping on $\mathbb{S}_{d}$ and $G(A)=\frac{1}{2}\mathbb{\hat{E}%
}[\langle AX,X\rangle]\leq \frac{1}{2}|A|\mathbb{\hat{E}}[|X|^{2}]$ implies
that there exists a bounded, convex and closed subset $\Gamma \subset
\mathbb{S}_{d}^{+}$ such that
\[
G(A)=\frac{1}{2}\sup_{\gamma \in \Gamma}\mathrm{tr}[\gamma A],
\]
where $\mathbb{S}_{d}^{+}$ denotes the collection of nonnegative elements in
$\mathbb{S}_{d}$.

\section{Characterizations of $G$-normal distribution}
We only consider non-degenerate random variable $X$ on a sublinear expectation space
$\left(\Omega,\mathcal{H},\hat{\mathbb{E}}\right)$, i.e. $\hat{\mathbb{E}}[X^2]>\left(\hat{\mathbb{E}}[|X|]\right)^2$. From the definition of $G$-normal distribution, an equivalent characterization of $G$-normal distribution is that,for any $a, b>0$
$$\frac{a}{\sqrt{a^{2}+b^{2}}}X+\frac{b}{\sqrt{a^{2}+b^{2}}}Y\overset{d}{=}X,$$
where $Y$ is an independent copy of $X$.  Denote $\lambda=\frac{a}{\sqrt{a^{2}+b^{2}}}$,
then
 $$\lambda X+\sqrt{1-\lambda^2}Y\overset{d}{=}X.$$
We are interested in the case that $\sqrt{1-\lambda^2}$ is replaced by $f(\lambda)$ which is a nonnegative function of $\lambda$. Actually we have the following theorem
\begin{theorem}\label{thm1}
Let $f$ be a nonnegative function defined on some interval of $\mathbb{R}$, which contains $0$ as an interior point and $X$ be a non-degenerate random variable on a sublinear expectation space
$\left(\Omega,\mathcal{H},\hat{\mathbb{E}}\right)$, for all $\lambda$ such that $f(\lambda)$ is non-negative,
$$\lambda X+f(\lambda)Y\overset{d}{=}X$$
where $Y$ is an independent copy of $X$, then:
\begin{description}
\item[(i)] $X$ is $G$-normal distributed;
\item[(ii)] $f(\lambda)=\sqrt{1-\lambda^2}$.
\end{description}
\end{theorem}
\begin{proof}
Denote $\hat{\mathbb{E}}[X]=\overline{\mu}$, $-\hat{\mathbb{E}}[-X]=\underline{\mu}$, $\hat{\mathbb{E}}[X^2]=\overline{\sigma}^2$ and $-\hat{\mathbb{E}}[-X^2]=\underline{\sigma}^2$, then
\begin{align}\label{1}
\hat{\mathbb{E}}\left[\lambda X+f(\lambda)Y\right]=f(\lambda)\overline{\mu}+\lambda^{+}\overline{\mu}-\lambda^{-}\underline{\mu}=\overline{\mu}.
\end{align}
\begin{align}\label{2}
-\hat{\mathbb{E}}\left[-\lambda X-f(\lambda)Y\right]=f(\lambda)\underline{\mu}-\lambda^{-}\overline{\mu}+\lambda^{+}\underline{\mu}=\underline{\mu}.
\end{align}
Without loss of generality, we assume that $\alpha$ and $-\alpha$ are in the domain of $f$ and $\alpha>0$. From \eqref{1} and \eqref{2}, we have
\begin{align}
&\label{a}f(-\alpha)\overline{\mu}-\alpha\underline{\mu}=\overline{\mu},\\
&\label{b}f(-\alpha)\underline{\mu}-\alpha\overline{\mu}=\underline{\mu}.
\end{align}
Hence,
\begin{equation*}
f(-\alpha)(\overline{\mu}-\underline{\mu})+\alpha(\overline{\mu}-\underline{\mu})=(\overline{\mu}-\underline{\mu}).
\end{equation*}
We shall prove $\overline{\mu}=\underline{\mu}$. Otherwise we suppose $\overline{\mu}\neq\underline{\mu}$. Thus $f(-\alpha)=1-\alpha$. Plugging into \eqref{a} and \eqref{b} yields $\overline{\mu}=-\underline{\mu}$. Hence for $\beta\in\mathbb{R}$,
\begin{equation*}
\hat{\mathbb{E}}[\beta X]=\beta^+\hat{\mathbb{E}}[X]+\beta^-\hat{\mathbb{E}}[-X]=|\beta|\hat{\mathbb{E}}[X].
\end{equation*}
Therefore
\begin{align*}
\overline{\sigma}^2&=\hat{\mathbb{E}}\left[\left(-\alpha X+f(-\alpha)Y\right)^2\right]\\
&\leq\hat{\mathbb{E}}\left[\alpha^2 X^2+f(-\alpha)^2Y^2\right]+2\hat{\mathbb{E}}\left[-\alpha f(-\alpha)XY\right]\\
&=\alpha^2\overline{\sigma}^2+f(-\alpha)^2\overline{\sigma}^2+2\alpha f(-\alpha)\overline{\mu}\hat{\mathbb{E}}\left[|X|\right]\\
&\leq\alpha^2\overline{\sigma}^2+f(-\alpha)^2\overline{\sigma}^2+2\alpha f(-\alpha)\left(\hat{\mathbb{E}}\left[|X|\right]\right)^2 .
\end{align*}
Hence
\begin{equation*}
2\alpha(1-\alpha)\overline{\sigma}^2\leq 2\alpha(1-\alpha)\left(\hat{\mathbb{E}}\left[|X|\right]\right)^2,
\end{equation*}
this contradiction yields that $\overline{\mu}=\underline{\mu}$. We now prove that $\overline{\mu}=\underline{\mu}=0$. Otherwise, from \eqref{1}, we have $f(\alpha)=1-\alpha$. Thus
\begin{align*}
\overline{\sigma}^2&=\hat{\mathbb{E}}\left[\left(\alpha X+f(\alpha)Y\right)^2\right]\\
&=\hat{\mathbb{E}}\left[\alpha^2 X^2+f(\alpha)^2Y^2\right]+2\hat{\mathbb{E}}\left[\alpha f(\alpha)XY\right]\\
&=\alpha^2\overline{\sigma}^2+f(\alpha)^2\overline{\sigma}^2+2\alpha f(\alpha)\overline{\mu}^2.
\end{align*}
Hence
\begin{equation*}
2\alpha(1-\alpha)\overline{\sigma}^2\leq 2\alpha(1-\alpha)\overline{\mu}^2,
\end{equation*}
this contradiction yields that $\overline{\mu}=\underline{\mu}=0$. Therefore
\begin{align*}
\hat{\mathbb{E}}\left[\left(\lambda X+f(\lambda)Y\right)^2\right]&=\lambda^2\overline{\sigma}^2+f(\lambda)^2\overline{\sigma}^2=\overline{\sigma}^2\\
-\hat{\mathbb{E}}\left[-\left(\lambda X+f(\lambda)Y\right)^2\right]&=\lambda^2\underline{\sigma}^2+f(\lambda)^2\underline{\sigma}^2=\underline{\sigma}^2.
\end{align*}
Hence, $f(\lambda)=\sqrt{1-\lambda^2}$. Thus $X$ is $G$-normal distributed.
\end{proof}
%

Moreover we can still have the following theorem without the condition that $Y$ and $X$ are identically distributed.
\begin{theorem}\label{thm2}
Let $X,Y$ be two non-degenerate random variables on a sublinear expectation space
$\left(\Omega,\mathcal{H},\hat{\mathbb{E}}\right)$and $f$ be a given non-negative function defined on some interval of $\mathbb{R}$, which contains $0$ as an interior point. Assuming that $Y$ is independent with $X$, and $\lambda X+f(\lambda)Y$ is a non-degenerate random variable whose distribution does not depend on $\lambda$ for all $\lambda$ such that $f(\lambda)$ is non-negative, then:
\begin{description}
\item[(i)] $f(\lambda)=\sqrt{a-b\lambda^2}$ for some $a,b>0$.
\item[(ii)] $X$ and $Y$ are $G$-normal distributed with $\overline{\sigma}^2_X=b\overline{\sigma}^2_Y,~\underline{\sigma}^2_X=b\underline{\sigma}^2_Y$.
\end{description}
\end{theorem}
\begin{proof}
Denote $\hat{\mathbb{E}}[X]=\overline{\mu}_X$, $-\hat{\mathbb{E}}[-X]=\underline{\mu}_X$, $\hat{\mathbb{E}}[X^2]=\overline{\sigma}^2_X$ and $-\hat{\mathbb{E}}[-X^2]=\underline{\sigma}^2_X$, $\hat{\mathbb{E}}[Y]=\overline{\mu}_Y$, $-\hat{\mathbb{E}}[-Y]=\underline{\mu}_Y$, $\hat{\mathbb{E}}[Y^2]=\overline{\sigma}^2_Y$ and $-\hat{\mathbb{E}}[-Y^2]=\underline{\sigma}^2_Y$, we have
\begin{align}\label{5}
\hat{\mathbb{E}}\left[\lambda X+f(\lambda)Y\right]=f(\lambda)\overline{\mu}_Y+\lambda^+\overline{\mu}_X-\lambda^{-}\underline{\mu}_X=\overline{h}.
\end{align}
\begin{align}\label{6}
-\hat{\mathbb{E}}\left[-\left(\lambda X+f(\lambda)Y\right)\right]=f(\lambda)\underline{\mu}_Y+\lambda^+\underline{\mu}_X-\lambda^{-}\overline{\mu}_X=\underline{h}.
\end{align}
Without loss of generality, we assume that $\alpha$ and $-\alpha$ are in the domain of $f$ and $\alpha>0$. From \eqref{5} and \eqref{6}, we have
\begin{align}
&\label{c}f(-\alpha)\overline{\mu}_Y-\alpha\underline{\mu}_X=\overline{h},\\
&\label{d}f(-\alpha)\underline{\mu}_Y-\alpha\overline{\mu}_X=\underline{h},\\
&\label{e}f(\alpha)\overline{\mu}_Y+\alpha\overline{\mu}_X=\overline{h},\\
&\label{f}f(\alpha)\underline{\mu}_Y+\alpha\underline{\mu}_X=\underline{h}.
\end{align}
Hence,
\begin{align*}
&f(\alpha)(\overline{\mu}_Y-\underline{\mu}_Y)+\alpha(\overline{\mu}_X-\underline{\mu}_X)=(\overline{h}-\underline{h})\\
&f(-\alpha)(\overline{\mu}_Y-\underline{\mu}_Y)+\alpha(\overline{\mu}_X-\underline{\mu}_X)=(\overline{h}-\underline{h}).
\end{align*}
\emph{Step 1.} We now prove that $\overline{\mu}_X=\underline{\mu}_X$ if and only if $\overline{\mu}_Y=\underline{\mu}_Y$. Since $\overline{h},\underline{h}$ are constants, we have $\overline{\mu}_X=\underline{\mu}_X$ if $\overline{\mu}_Y=\underline{\mu}_Y$. We suppose that $\overline{\mu}_Y\neq\underline{\mu}_Y$ if $\overline{\mu}_X=\underline{\mu}_X$. Thus $f(\alpha)=f(-\alpha)$. From \eqref{c} and \eqref{e}, we obtain $\overline{\mu}_X=\underline{\mu}_X=0$. Therefore $f$ is given by
\begin{equation*}
f(\lambda)=\frac{\overline{h}-\underline{h}}{\overline{\mu}_Y-\underline{\mu}_Y},
\end{equation*}
and the domain of $f$ is $\mathbb{R}$. Moreover, we have
\begin{align*}
\overline{\sigma}^2&=\hat{\mathbb{E}}\left[(\lambda X+f(\lambda)Y)^2\right]\\
&\geq \lambda^2\overline{\sigma}^2_X+\left(\frac{\overline{h}-\underline{h}}{\overline{\mu}_Y-\underline{\mu}_Y}\right)^2\overline{\sigma}^2_Y-2\lambda\frac{\overline{h}-\underline{h}}{\overline{\mu}_Y-\underline{\mu}_Y}\hat{\mathbb{E}}[-XY]\\
&\lambda^2\overline{\sigma}^2_X+\left(\frac{\overline{h}-\underline{h}}{\overline{\mu}_Y-\underline{\mu}_Y}\right)^2\overline{\sigma}^2_Y-2\lambda\frac{\overline{h}-\underline{h}}{\overline{\mu}_Y-\underline{\mu}_Y}\left(\hat{\mathbb{E}}[X^2]\right)^{\frac{1}{2}}\left(\hat{\mathbb{E}}[Y^2]\right)^{\frac{1}{2}}\\
&=\left(\lambda\sqrt{\overline{\sigma}^2_X}-\frac{\overline{h}-\underline{h}}{\overline{\mu}_Y-\underline{\mu}_Y}\sqrt{\overline{\sigma}^2_Y}\right)^2\stackrel{\lambda\rightarrow\infty}{\longrightarrow}\infty.
\end{align*}
Since $\overline{\sigma}^2$ is a constant, this contradiction implies that $\overline{\mu}_Y=\underline{\mu}_Y$.\\
\emph{Step 2.} We shall prove $\overline{\mu}_Y=\underline{\mu}_Y$, $\overline{\mu}_X=\underline{\mu}_X$. Otherwise we suppose $\overline{\mu}_Y\neq\underline{\mu}_Y$ (hence $\overline{\mu}_X\neq\underline{\mu}_X$). Thus
 \begin{equation*}
 f(\alpha)=f(-\alpha)=\frac{\overline{h}-\underline{h}}{\overline{\mu}_Y-\underline{\mu}_Y}-\alpha\frac{\overline{\mu}_X-\underline{\mu}_X}{\overline{\mu}_Y-\underline{\mu}_Y}.
 \end{equation*}
 Plugging into \eqref{c} and \eqref{d} yields
 \begin{align*}
 &\frac{\overline{h}-\underline{h}}{\overline{\mu}_Y-\underline{\mu}_Y}\overline{\mu}_Y-\alpha\left(\frac{\overline{\mu}_X-\underline{\mu}_X}{\overline{\mu}_Y-\underline{\mu}_Y}\overline{\mu}_Y+\underline{\mu}_X\right)=\overline{h},\\
 &\frac{\overline{h}-\underline{h}}{\overline{\mu}_Y-\underline{\mu}_Y}\underline{\mu}_Y-\alpha\left(\frac{\overline{\mu}_X-\underline{\mu}_X}{\overline{\mu}_Y-\underline{\mu}_Y}\underline{\mu}_Y+\overline{\mu}_X\right)=\underline{h}.
 \end{align*}
 Hence
 \begin{align*}
 &\underline{\mu}_X=-\frac{\overline{\mu}_X-\underline{\mu}_X}{\overline{\mu}_Y-\underline{\mu}_Y}\overline{\mu}_Y,\\
 &\overline{\mu}_X=-\frac{\overline{\mu}_X-\underline{\mu}_X}{\overline{\mu}_Y-\underline{\mu}_Y}\underline{\mu}_Y.
 \end{align*}
 Since $f(\alpha)=f(-\alpha)$, from \eqref{c} and \eqref{e}, we obtain $\overline{\mu}_X=-\underline{\mu}_X$. Thus $\overline{\mu}_Y=-\underline{\mu}_Y$, $\overline{h}=-\underline{h}$. Therefore for $\beta\in\mathbb{R}$,
\begin{align*}
\hat{\mathbb{E}}[\beta X]=\beta^+\hat{\mathbb{E}}[X]+\beta^-\hat{\mathbb{E}}[-X]=|\beta|\hat{\mathbb{E}}[X],\\
\hat{\mathbb{E}}[\beta Y]=\beta^+\hat{\mathbb{E}}[Y]+\beta^-\hat{\mathbb{E}}[-Y]=|\beta|\hat{\mathbb{E}}[Y].
\end{align*}
If $\overline{\mu}_Y=0$, we have $\overline{\mu}_Y=\underline{\mu}_Y=\overline{\mu}_X=\underline{\mu}_X=\overline{h}=\underline{h}=0$. Now we suppose that $\overline{\mu}_Y=-\underline{\mu}_Y\neq 0$. Hence $\overline{\mu}_X=-\underline{\mu}_X\neq 0$. Therefore with $\frac{\overline{h}}{\overline{\mu}_Y}>0$, $f$ is given by
\begin{equation*}
f(\lambda)=\frac{\overline{h}}{\overline{\mu}_Y}-|\lambda|\frac{\overline{\mu}_X}{\overline{\mu}_Y}.
\end{equation*}
If $\frac{\overline{\mu}_X}{\overline{\mu}_Y}<0$, the domain of $f$ is $\mathbb{R}$. Therefore,
\begin{align*}
\overline{\sigma}^2&=\hat{\mathbb{E}}\left[(\lambda X+f(\lambda)Y)^2\right]\\
&\geq \lambda^2\overline{\sigma}^2_X+\left(\frac{\overline{h}}{\overline{\mu}_Y}-|\lambda|\frac{\overline{\mu}_X}{\overline{\mu}_Y}\right)^2\overline{\sigma}^2_Y-2\lambda\left(\frac{\overline{h}}{\overline{\mu}_Y}-|\lambda|\frac{\overline{\mu}_X}{\overline{\mu}_Y}\right)\hat{\mathbb{E}}[-XY]\\
&\geq\lambda^2\overline{\sigma}^2_X+\left(\frac{\overline{h}}{\overline{\mu}_Y}-|\lambda|\frac{\overline{\mu}_X}{\overline{\mu}_Y}\right)^2\overline{\sigma}^2_Y-2\lambda\left(\frac{\overline{h}}{\overline{\mu}_Y}-|\lambda|\frac{\overline{\mu}_X}{\overline{\mu}_Y}\right)\left(\hat{\mathbb{E}}[X^2]\right)^{\frac{1}{2}}\left(\hat{\mathbb{E}}[Y^2]\right)^{\frac{1}{2}}\\
&=\left(\lambda\sqrt{\overline{\sigma}^2_X}-\left(\frac{\overline{h}}{\overline{\mu}_Y}-|\lambda|\frac{\overline{\mu}_X}{\overline{\mu}_Y}\right)\sqrt{\overline{\sigma}^2_Y}\right)^2\stackrel{\lambda\rightarrow -\infty}{\longrightarrow}\infty.
\end{align*}
This contradiction implies that $\overline{\mu}_Y=\underline{\mu}_Y$, $\overline{\mu}_X=\underline{\mu}_X$.\\
If $\frac{\overline{\mu}_X}{\overline{\mu}_Y}>0$, the domain of $f$ is $[-\frac{\overline{h}}{\overline{\mu}_X},\frac{\overline{h}}{\overline{\mu}_X}]$. Hence $\overline{h}\neq 0$. Therefore
\begin{equation*}
\overline{\sigma}^2=\hat{\mathbb{E}}\left[\left(\frac{\overline{h}}{\overline{\mu}_X}X\right)^2\right]=\hat{\mathbb{E}}\left[\left(\frac{\overline{h}}{\overline{\mu}_Y}Y\right)^2\right],
\end{equation*}
from which we obtain $\overline{\sigma}^2_X=\frac{\overline{\mu}^2_X}{\overline{h}^2}\overline{\sigma}^2$ and $\overline{\sigma}^2_Y=\frac{\overline{\mu}^2_Y}{\overline{h}^2}\overline{\sigma}^2$. Thus
\begin{align*}
\hat{\mathbb{E}}\left[\left(\alpha X+f(\alpha)Y\right)^2\right]&\leq \hat{\mathbb{E}}\left[\alpha^2 X^2+f(\alpha)^2Y^2\right]+2\alpha f(\alpha)E[XY]\\
&=\alpha^2\overline{\sigma}^2_X+f(\alpha)^2\overline{\sigma}^2_Y+2\alpha f(\alpha)\overline{\mu}_Y\hat{\mathbb{E}}\left[|X|\right]\\
&<\alpha^2\overline{\sigma}^2_X+f(\alpha)^2\overline{\sigma}^2_Y+2\alpha f(\alpha)\left(\hat{\mathbb{E}}\left[|Y|^2\right]\right)^{\frac{1}{2}}\left(\hat{\mathbb{E}}\left[|X|^2\right]\right)^{\frac{1}{2}}\\
&=\left(\alpha\sqrt{\overline{\sigma}^2_X}+f(\alpha)\sqrt{\overline{\sigma}^2_Y}\right)^2\\
&=\left(\alpha\frac{\overline{\mu}_X}{\overline{h}}+1-\alpha\frac{\overline{\mu}_X}{\overline{h}}\right)^2\overline{\sigma}^2=\overline{\sigma}^2.
\end{align*}
This contradiction implies that $\overline{\mu}_Y=\underline{\mu}_Y$, $\overline{\mu}_X=\underline{\mu}_X$.\\
\emph{Step 3.} We shall prove $\overline{\mu}_Y=\underline{\mu}_Y=0$, hence $\overline{\mu}_X=\underline{\mu}_X=0$. Otherwise suppose $\overline{\mu}_Y\neq 0$. Thus
\[
f(\lambda)=\frac{\overline{h}}{\overline{\mu}_Y} -\lambda\frac{\overline{\mu}_X}{\overline{\mu}_Y}.
 \]
Hence,
\begin{align*}
\overline{\sigma}^2&=\hat{\mathbb{E}}\left[(\lambda X+f(\lambda)Y)^2\right]\\
&=\lambda^2\overline{\sigma}^2_X+f(\lambda)^2\overline{\sigma}^2_X+2\lambda f(\lambda)\overline{\mu}_X\overline{\mu}_Y\\
&=\lambda^2\left(\overline{\sigma}^2_X+\frac{\overline{\mu}^2_X}{\overline{\mu}^2_Y}\overline{\sigma}^2_Y-2\overline{\mu}^2_X\right)-2\lambda\overline{h}\overline{\mu}_X\left(\frac{\overline{\sigma}^2_Y}{\overline{\mu}^2_Y}-1\right)+\frac{\overline{h}^2}{\overline{\mu}^2_Y}\overline{\sigma}^2_Y.
\end{align*}
depends on $\lambda$.
This contraction yields that $\overline{\mu}_Y=0$. Therefore
\begin{align*}
\overline{\sigma}^2=\lambda^2\overline{\sigma}^2_X+f(\lambda)\overline{\sigma}^2_Y,
\end{align*}
\begin{align*}
\underline{\sigma}^2=\lambda^2\underline{\sigma}^2_X+f(\lambda)\underline{\sigma}^2_Y.
\end{align*}
We have $f(\lambda)=\sqrt{\frac{\overline{\sigma}^2}{\overline{\sigma}^2_Y}-\frac{\overline{\sigma}^2_X}{\overline{\sigma}^2_Y}\lambda^2}=\sqrt{a-b\lambda^2}$
where $a=\frac{\overline{\sigma}^2}{\overline{\sigma}^2_Y}$, $b=\frac{\overline{\sigma}^2_X}{\overline{\sigma}^2_Y}$ and $\underline{\sigma}^2=a\underline{\sigma}^2_Y,~\underline{\sigma}^2_X=b\underline{\sigma}^2_Y$.\\
The domain of $f(\lambda)$ is the interval $\left[-\sqrt{\frac{a}{b}},\sqrt{\frac{a}{b}}\right]$.\\
For $\lambda=0$ or $\lambda=\sqrt{\frac{a}{b}}$, the two random variables $\sqrt{a}Y$ and $\sqrt{\frac{a}{b}}X$ are identically distributed according to
\[
\lambda X+\sqrt{a-b\lambda^2}Y\overset{d}{=}\left(\sqrt{1-\frac{b}{a}\lambda^2}\right)\left(\sqrt{\frac{a}{b}}X\right)+\left(\lambda\sqrt{\frac{b}{a}}\right)\left(\sqrt{a}Y\right)
\]
By Theorem (\ref{thm1}),$\sqrt{\frac{a}{b}}X$ and $\sqrt{a}Y$ are $G$-normal distributed. Hence $X$ and $Y$ are $G$-normal distributed.
\end{proof}
\begin{remark}
Note that the distribution uncertainty of a $G$-normal distribution is not just the parameter of the classical normal distributions, it is difficult to obtain characterization results of $G$-normal distribution. Let $g_{\lambda}(u,v)=\lambda u+f(\lambda)v$ be a family of functions from $\mathbb{R}\times\mathbb{R}\rightarrow\mathbb{R}$. With the nondependence of the distribution of $g_{\lambda}(X,Y)$ upon $\lambda$, our results concern which possible forms of $g_{\lambda}(u,v)$ provide $G$-normality of $X$ and $Y$. In conclusion , our results complement the characterizations of $G$-normal distribution.
\end{remark}
\section*{Acknowledgment}


\end{document}